\documentclass[12pt]{article}

\usepackage{tikz}
\usetikzlibrary{decorations.markings}

\usepackage{amsmath}
\usepackage{amsfonts}
\usepackage{amssymb}
\usepackage{amsthm}
\usepackage{verbatim}
\usepackage{enumerate}
\usepackage{url}
\usepackage{geometry}

\newtheorem{theorem}{Theorem}
\newtheorem{lemma}[theorem]{Lemma}

\theoremstyle{definition}
\newtheorem{problem}[theorem]{Problem}
\theoremstyle{remark}
\newtheorem{claim}{Claim}
\newtheorem{claimx}{Claim}

\title{Hall's Condition for Partial Latin Squares}
\author{A. J. W. Hilton \thanks{Department of Mathematics and Statistics, University of Reading, Whiteknights, Reading RG6 6AX, U.K. \textit{and} School of Mathematical Sciences, Queen Mary, University of London, Mile End Road, London E1 4NS, U.K. Email: {\tt a.j.w.hilton@reading.ac.uk}}
\and E. R. Vaughan \thanks{School of Mathematical Sciences, Queen Mary, University of London, Mile End Road, London E1 4NS, U.K. Email: {\tt e.vaughan@qmul.ac.uk}}}
\date{\today}

\begin{document}

\maketitle

\providecommand{\size}[1]{\left|#1\right|}
\providecommand{\floor}[1]{\lfloor#1\rfloor}
\providecommand{\ceiling}[1]{\lceil#1\rceil}

\providecommand{\hc}{Hall's Condition}
\providecommand{\hil}{Hall Inequality}
\providecommand{\his}{Hall's Inequalities}
\providecommand{\HI}[1]{\ensuremath{\mathcal{H} (#1)}}

\begin{abstract} \hc\ is a necessary condition for a partial latin square to be completable. Hilton and Johnson showed that for a partial latin square whose filled cells form a rectangle, \hc\ is equivalent to Ryser's Condition, which is a necessary and sufficient condition for completability.

We give what could be regarded as an extension of Ryser's Theorem, by showing that for a partial latin square whose filled cells form a rectangle, where there is at most one empty cell in each column of the rectangle, \hc\ is a necessary and sufficient condition for completability.

It is well-known that the problem of deciding whether a partial latin square is completable is NP-complete. We show that the problem of deciding whether a partial latin square that is promised to satisfy \hc\ is completable is NP-hard. \end{abstract}

\section{Introduction} \label{sec:introduction}

The following is a classical result of Ryser, which provides a necessary and sufficient condition for completability of partial latin squares where the filled cells form a rectangle. In this and subsequent results, we shall assume that the rectangle is in the upper left corner of the square. This is merely for convenience, and in fact Theorem \ref{ryser} holds for any partial latin square that can be put into this form by permuting its rows and columns.

\begin{theorem} \emph{(Ryser, 1951 \cite{bipartite, ryser})} Let $P$ be a partial latin square of order $n$ whose filled cells are those in the upper left $r\times s$ rectangle $R$, for some $r,s \in \{1, \dots, n\}$. Then $P$ is completable if and only if \[\nu(\sigma) \ge r + s - n\] for each symbol $\sigma \in \{1, \dots, n\}$, where $\nu(\sigma)$ is the number of times that $\sigma$ appears in $R$. \label{ryser} \end{theorem}

This set of $n$ inequalities as known as \emph{Ryser's Condition}. In this paper, we consider another condition, introduced by Hilton and Johnson, that is known as \textit{\hc}\ \cite{hiltonjohnson2}. Bobga and Johnson \cite{bobgaphd, bobgajohnson} observed that for partial latin squares where the filled cells form a rectangle, \hc\ is equivalent to Ryser's Condition.\footnote{In fact, this follows from a result of Hilton and Johnson \cite{hiltonjohnson}, who show that the $n$ inequalities of Ryser's Condition can be replaced by a single inequality. They did not state their result in terms of \hc, but they clearly realized this very quickly.}

To state \hc\, we first need some definitions. Given a partial latin square, a symbol is said to be \textit{missing} from a row (or column) if it does not appear in a filled cell of that row (or column). Given a symbol $\sigma \in \{1, \dots, n\}$, a cell is said to \textit{support} $\sigma$ if the cell either contains $\sigma$, or the cell is empty and $\sigma$ is missing from the cell's row and column. A set of cells is said to be \textit{independent} if no two of the cells belong to the same row or column; and if each cell in the set supports $\sigma$, the set is said to be an \textit{independent set for $\sigma$}.

Let $P$ be a partial latin square of order $n$. Given a set $T$ of cells of $P$, and a symbol $\sigma \in \{1, \dots, n\}$, let $\alpha(\sigma, T)$ denote the size of the largest subset of $T$ that is an independent set for $\sigma$. Then the \textit{\hil}\ for $T$, denoted \HI{T}, is the inequality
\begin{equation*}
\sum_{\sigma = 1}^n \alpha(\sigma, T) \ge \size{T}. \label{plshalls}
\end{equation*}
The partial latin square $P$ is said to satisfy \textit{\hc}\ if for each set $T$ of cells of $P$, the \hil\ \HI{T} is satisfied.\footnote{\hc\ can actually be defined in a more general setting; namely graphs whose vertices are equipped with colour lists. See \cite{hoffmanjohnson} for a survey.}

It is not hard to show that \hc\ is a \textit{necessary} condition for completability of partial latin squares (see Lemma~\ref{hconlyif}). Some time ago Cropper asked whether in fact \hc\ is a \textit{sufficient} condition \cite{cropper}. John Goldwasser provided a negative answer, giving a partial latin square (see Figure~\ref{goldwasser}) that satisfies \hc\ but is not completable \cite{ghp}. However, Cropper's question served to stimulate interest in the area, and a number of papers have appeared recently (e.g. \cite{bobgaphd, bobgajohnson, ghp, hiltonvaughan}).

\begin{figure}
\begin{center} \begin{tikzpicture}
[matrix of nodes/.style={minimum size=7mm, execute at begin cell=\node\bgroup, execute at end cell=\egroup; }]
\draw[color=gray,step=7mm] (0,0) grid (4.2,4.2);
\node [matrix, matrix of nodes] at (2.1,2.1) {
1 & 2 & 3 & 4 & 5 & 6 \\
3 & 6 & 1 & 2 & 4 & 5 \\
5 & 4 & 2 & 6 & 3 & 1 \\
2 & 5 &  &  &  &  \\
4 & 1 &  &  &  &  \\
6 & 3 &  &  &  &  \\ };
\end{tikzpicture} \end{center}
\caption{Goldwasser's square} \label{goldwasser} \end{figure}

As noted above, in the case of partial latin squares where the filled cells form a rectangle, \hc\ is both a necessary and sufficient condition for completability. So we have the following theorem.

\begin{theorem} Let $P$ be a partial latin square of order $n$ whose filled cells are those in the upper left $r\times s$ rectangle, for some $r,s \in \{1, \dots, n\}$. Then $P$ is completable if and only if $P$ satisfies \hc. \label{pls1} \end{theorem}

So while \hc\ is not, in general, a sufficient condition for a partial latin square to be completable, if we restrict our attention to partial latin squares where the filled cells form a rectangle, it is both a necessary and sufficient condition for completability. In light of Theorem~\ref{pls1} it seems sensible to ask if there are other classes of partial latin square for which \hc\ is also a sufficient condition for completability.

Bobga and Johnson considered this question, and found that \hc\ is also a sufficient condition for completability in the case of partial latin squares where the filled cells form a rectangle with one empty cell inside \cite{bobgaphd, bobgajohnson}.
In Section \ref{sec:ryser}, we shall prove the following generalization of their result.

\begin{theorem} Let $P$ be a partial latin square of order $n$ whose filled cells are those in the upper left $r\times s$ rectangle, for some $r,s \in \{1, \dots, n\}$, except for $t$ cells in this rectangle that are empty, with the condition that there is no more than one of these empty cells in each column. Then $P$ is completable if and only if $P$ satisfies \hc. \label{pls5} \end{theorem}


To the extent that Theorem~\ref{pls1} is a restatement of Ryser's Theorem (Theorem~\ref{ryser}), Theorem~\ref{pls5} can be considered as a generalization of Ryser's Theorem. It seems likely that Theorem~\ref{pls5} is not the end of the matter, and that more general results along these lines are possible. For example, one could consider partial latin squares where the filled cells form a rectangle, except for a $2 \times t$ rectangle of empty cells inside. We do not know if \hc\ is a sufficient condition for completability in this case. 


In Section \ref{sec:complexity}, we shall consider \hc\ and computational complexity. In this section we use some standard notions from theoretical computer science, for which we refer the reader to the book by Garey and Johnson \cite{gareyjohnson}. An obvious computational problem to study is the following.

\begin{problem} Let $P$ be a partial latin square. Decide if $P$ satisfies \hc. \label{hcprob} \end{problem}

Unfortunately we are not able to say much about the complexity of this problem. It may be the case that Problem~\ref{hcprob} is in P, but at present we cannot even show that it is in NP. In Section \ref{sec:preliminaries}, we show that it is possible to check each \hil\ in polynomial time (see Lemma~\ref{hilpoly}), and so if we have a partial latin square that does \textit{not} satisfy \hc, this fact can be verified in polynomial time. Thus Problem~\ref{hcprob} is in co-NP.

The following problem was shown to be NP-complete by Colbourn \cite{colbourn} (see also \cite{eastonparker}).

\begin{problem} Let $P$ be a partial latin square. Decide if $P$ is completable. \label{plsprob} \end{problem}

We can consider the following variant of Problem~\ref{plsprob}.

\begin{problem} Let $P$ be a partial latin square that satisfies \hc. Decide if $P$ is completable. \label{plshcprob} \end{problem}

The set of partial latin squares that are ``yes'' instances of Problem~\ref{plsprob} is the same as the set of partial latin squares that are ``yes'' instances of Problem~\ref{plshcprob}. The difference is that in Problem~\ref{plshcprob}, the input is restricted to partial latin squares that satisfy \hc. Thus Problem~\ref{plshcprob} is an example of a \textit{promise problem}, as the input partial latin square is ``promised'' to satisfy \hc. (See \cite{goldreich} for a survey of promise problems.) We shall prove the following.

\begin{theorem} Problem~\ref{plshcprob} is NP-hard. \label{plsnph} \end{theorem}

Theorem~\ref{plsnph} suggests that knowing that a partial latin square satisfies \hc\ may not be very helpful in determining its completability. In Section \ref{sec:complexity} we shall give a reduction from an NP-complete hypergraph colouring problem to Problem~\ref{plsprob}, with the property that its image is contained in the set of partial latin squares that satisfy \hc. From the existence of this reduction, we can deduce that Problem~\ref{plshcprob} is NP-hard, and obtain a new proof that Problem~\ref{plsprob} is NP-complete.


The structure of this paper is as follows. In Section~\ref{sec:preliminaries} we prove some basic results and quote some classical theorems. In Section~\ref{sec:ryser} we give the proof of Theorem~\ref{pls5} and in Section~\ref{sec:complexity} we give the proof of Theorem~\ref{plsnph}.

\section{Preliminaries} \label{sec:preliminaries}

In this section we give some basic results that will be needed later. First we give the standard result that \hc\ is a necessary condition for a partial latin square to be completable.

\begin{lemma} Let $P$ be a partial latin square of order $n$. Then $P$ is completable only if $P$ satisfies \hc. \label{hconlyif} \end{lemma}

\begin{proof} Let $P^*$ be a completion of $P$, and let $T$ be a set of cells of $P^*$. For each symbol $\sigma \in \{1, \dots, n\}$, let $T_\sigma$ be the subset of the cells of $T$ that contain $\sigma$. Summing over all symbols, we have $T_1 + \dots + T_n = \size{T}$. As each set of cells $T_\sigma$ is an independent set, we have $\alpha(\sigma, T) \ge \size{T_\sigma}$. Therefore
\[ \sum_{\sigma=1}^n \alpha(\sigma, T) \ge  \sum_{\sigma=1}^n \size{T_\sigma} = \size{T}, \] and so \HI{T} holds. Since this is true for all sets $T$, it follows that $P$ satisfies \hc. \end{proof}

We also need the following standard lemma.

\begin{lemma} Let $P$ be a partial latin square of order $n$, $T$ a set of cells, and $f \in T$ a filled cell. Then \HI{T} holds if and only if \HI{T-f} holds. \label{t-f} \end{lemma}

\begin{proof} Assume that \HI{T} holds. As the cell $f$ is filled, it only supports one symbol, and so it can only contribute 1 to the quantity
\[\sum_{\sigma=1}^n \alpha(\sigma, T), \]
and so
\[ \sum_{\sigma=1}^n \alpha(\sigma, T-f) \ge \left( \sum_{\sigma=1}^n \alpha(\sigma, T) \right) - 1 \ge \size{T} - 1 = \size{T-f}, \]
and so \HI{T-f} holds. Conversely, suppose that \HI{T-f} holds, and that $f$ contains a symbol $\sigma \in \{1, \dots, n\}$. As no other cells of $T$ in the same row or column as $f$ support $\sigma$, the size of a maximum independent set of cells of $T$ that support $\sigma$ is exactly one greater than the size of such a set in $T-f$. So
\[ \sum_{\sigma=1}^n \alpha(\sigma, T) = \left( \sum_{\sigma=1}^n \alpha(\sigma, T-f) \right) + 1 \ge \size{T-f} + 1 = \size{T}, \]
and so \HI{T} holds. \end{proof}

Because of Lemma~\ref{t-f}, if one wishes to determine if a partial latin square satisfies \hc, it is sufficient to verify that the \hil\ for each set of \textit{empty} cells is satisfied. So we have the following theorem.

\begin{theorem} Let $P$ be a partial latin square. $P$ satisfies \hc\ if and only if the \hil\ is satisfied by each set of empty cells.
\label{empty} \end{theorem}

A \textit{vertex cover} of a graph $G$ is a set of vertices $C$, such that $C$ contains at least one vertex of each edge of $G$. We shall need the following classical theorem.

\begin{theorem} \emph{(K\"onig--Egerv\'ary, 1931 \cite{bipartite, bondymurty, diestel, konig})} Let $G$ be a bipartite graph. The size of a maximum matching in $G$ is equal to the size of a minimum vertex cover. \label{konig-egervary} \end{theorem}

It can be easily verified that Goldwasser's square (see Figure~\ref{goldwasser}) is incompletable. To see that it satisfies \hc, we can use the following lemma, which will also prove useful in Section \ref{sec:complexity}.

\begin{lemma} Let $P$ be a partial latin square. Suppose that a symbol $\sigma$ is missing from $k$ columns and $k$ rows of $P$ but is present in all the other rows and columns of $P$, and that we have a set of empty cells $T$ which contains $t$ cells that support $\sigma$. Then $\alpha(\sigma, T) \ge \ceiling{t/k}$. \label{toverk} \end{lemma}

\begin{proof} Let $G$ be the bipartite graph on $2k$ vertices, defined as follows. There are $k$ vertices $r_1, \dots, r_k$, representing the rows from which $\sigma$ is missing, and $k$ vertices $c_1, \dots, c_k$, representing the columns from which $\sigma$ is missing. For each cell of $T$ that supports $\sigma$, we place an edge between the vertex that represents the cell's row, and the vertex that represents the cell's column. Thus matchings in $G$ correspond to independent sets of cells in $P$. By Theorem~\ref{konig-egervary}, the size of a maximum matching in $G$ is equal to the size of the smallest vertex cover. The maximum degree of $G$ is at most $k$, so a set of $s$ vertices is incident with at most $sk$ edges. Since $G$ has $t$ edges, a vertex cover of $G$ must contain at least $\ceiling{t/k}$ vertices. Hence there is a matching in $G$ of size $\ceiling{t/k}$, and therefore $\alpha(\sigma, T) \ge \ceiling{t/k}$. \end{proof}

Note that in the preceding proof we observed that $\alpha(\sigma, T)$ is equal to the size of a maximum matching in a certain bipartite graph. So we can determine if \HI{T} holds by computing the size of a maximum matching in $n$ bipartite graphs, one for each choice of symbol $\sigma$. Since the size of a maximum matching in a bipartite graph on $N$ vertices can be determined in $O(N^3)$ time (see e.g. \cite[Chapter 20]{schrijver}), we have the following result.

\begin{lemma} Let $P$ be a partial latin square, and $T$ a set of cells. Then \HI{T} can be determined in $O(n^4)$ time. \label{hilpoly}
\end{lemma}

In Goldwasser's square, each symbol is missing from 2 rows and 2 columns, and each empty cell supports 2 symbols. Let $T$ be a set of empty cells of Goldwasser's square. For each $i \in \{1, \dots, 6\}$, let $a_i$ be the number of cells of $T$ that support $i$. As each cell of $T$ supports 2 symbols we have \begin{equation} \sum_{i=1}^6 a_i = 2\size{T} \label{2T} \end{equation} But then \begin{align*}
\sum_{i=1}^6 \alpha(i, T) &\ge \sum_{i=1}^6 \ceiling{a_i / 2} &&\text{(by Lemma~\ref{toverk})} \\
&\ge \tfrac{1}{2} \sum_{i=1}^6 a_i \\ &= \size{T}, &&\text{(by (\ref{2T}))} 
\end{align*} and so \HI{T} holds. By Theorem~\ref{empty}, \hc\ holds if the \hil\ holds for each set of empty cells. Since this holds for all sets $T$, it follows that Goldwasser's square satisfies \hc. The following theorem formalizes this argument.

\begin{theorem} Let $P$ be a partial latin square of order $n$. For all $\sigma \in \{1, \dots, n\}$, let  $\nu(\sigma)$ denote the number of times that $\sigma$ appears in $P$. For each empty cell $b$ of $P$ we let $S(b)$ denote the set of symbols supported by $b$. Suppose we have
\begin{equation}  \sum_{\sigma \in S(b)} \frac{1}{n - \nu(\sigma)} \ge 1 \; \; \text{for each empty cell $b$ of $P$.} \label{atleast1} \end{equation}
Then $P$ satisfies \hc. \label{atleast1t} \end{theorem}	

\begin{proof} By Theorem~\ref{empty}, \hc\ holds if the \hil\ holds for each set of empty cells. Let $T$ be a set of empty cells of $P$. For each symbol $\sigma \in \{1, \dots, n\}$, let $T_\sigma$ denote the subset of $T$ consisting of the cells that support $\sigma$. Then
\begin{align*}
\sum_{\sigma = 1}^n \alpha(\sigma, T) &\ge
\sum_{\sigma = 1}^n \frac{\size{T_\sigma}}{n - \nu(\sigma)} &&\text{(by Lemma~\ref{toverk})} \\
&= \sum_{b \in T} \sum_{\sigma \in S(b)} \frac{1}{n - \nu(\sigma)} \\
&\ge \sum_{b \in T} 1 &&\text{(by (\ref{atleast1}))} \\
&= \size{T},
\end{align*} and so \HI{T} holds.
\end{proof}

Finally, we state four classical results that will be needed in due course.

\begin{theorem} \emph{(Hall, 1936 \cite{bipartite, bondymurty, diestel, hall})} Let $G$ be a bipartite graph with bipartition $(A, B)$. There is a matching in $G$ which covers $A$ if and only if each subset $A' \subseteq A$ has at least $\size{A'}$ neighbours. \label{hall2} \end{theorem}

\begin{theorem} \emph{(Dulmage and Mendelsohn, 1958 \cite{bipartite, dulmagemendelsohn})} Let $G$ be a bipartite graph with bipartition $(A, B)$ and suppose $M_1$ and $M_2$ are two matchings in $G$. Then there is a matching $M \subseteq M_1 \cup M_2$ such that $M$ covers all the vertices of $A$ covered by $M_1$ and all the vertices of $B$ covered by $M_2$. \label{dulmage} \end{theorem}

For a graph $G$, the maximum degree will be denoted $\Delta(G)$ and the chromatic index will be denoted $\chi'(G)$.

\begin{theorem} \emph{(K\"onig, 1916 \cite{bipartite, bondymurty, diestel, konig})} Let $G$ be a bipartite graph. Then $\chi'(G) = \Delta(G)$. \label{konig} \end{theorem}

The final theorem, of Ford and Fulkerson, is commonly known as the \textit{Max-flow Min-cut Theorem}. An \textit{integral flow} is a flow in which the flow along each edge is an integer.

\begin{theorem} \emph{(Ford and Fulkerson, 1956 \cite{bondymurty, diestel, fordfulkerson})} Let $G$ be a directed graph with integral edge capacities and two distinguished vertices $\alpha$ (the ``source'') and $\omega$ (the ``sink''). Then the size of a maximum flow between $\alpha$ and $\omega$ is equal to the minimum size of a cut that separates these two vertices. Moreover, there is a maximum flow that is integral. \label{maxflow} \end{theorem}

It is a common student exercise to deduce Theorem~\ref{hall2} (Hall's Theorem) from Theorem~\ref{maxflow}. We shall use a method in the proof of Theorem~\ref{pls5} that was inspired by this exercise.

\section{A generalization of Ryser's Theorem} \label{sec:ryser}

Theorem~\ref{pls1} states that Ryser's Condition is equivalent to \hc. In this section, we give a proof of Theorem~\ref{pls1}, in the same spirit as that given by Bobga and Johnson \cite{bobgaphd, bobgajohnson}. The proof will serve as a template for the more difficult proof of Theorem~\ref{pls5}. In the proof we show that Ryser's Condition is in fact equivalent to a single \hil: \HI{H}, where $H$ is the set of cells in the top $r$ rows.
To prove Theorem~\ref{pls1} we need the following lemma.

\begin{lemma} Let $P$ be a partial latin square of order $n$ whose filled cells are those in the upper left $r\times s$ rectangle $R$, for some $r,s \in \{1, \dots, n\}$. Let $H$ be the set of cells in the top $r$ rows, and for each symbol $\sigma \in \{1, \dots, n\}$ let $\nu(\sigma)$ denote the number of times that $\sigma$ appears in $R$. Then
\[ \alpha(\sigma, H) = \min\{r, \ \nu(\sigma)+n-s\}. \]
\label{alphasigma} \end{lemma}

\begin{proof} Fix a symbol $\sigma \in \{1, \dots, n\}$. Let $S_1$ be the set of cells in $R$ that contain $\sigma$. $S_1$ is an independent set of size $\nu(\sigma)$. In $H$ there are $n-s$ columns with empty cells, and $\sigma$ is missing from $r-\nu(\sigma)$ rows. From the cells in these $n-s$ columns and $r-\nu(\sigma)$ rows we can select an independent set $S_2$ of size $\min\{r-\nu(\sigma), \ n-s\}$. Then $S = S_1 \cup S_2$ is an independent set of size $\min\{r, \ \nu(\sigma)+n-s\}$. So $\alpha(\sigma, H) \ge \min\{r, \ \nu(\sigma)+n-s\}$. Moreover, $\alpha(\sigma, H) \le r$ as there are $r$ rows in $H$; also $\alpha(\sigma, H) \le \nu(\sigma)+n-s$ as no independent set for $\sigma$ can use more than $\nu(\sigma)+n-s$ columns. So in fact we have $\alpha(\sigma, H) = \min\{r, \ \nu(\sigma)+n-s\}$. \end{proof}

We can now prove Theorem~\ref{pls1}.

\begin{proof}[Proof of Theorem~\ref{pls1}] If $P$ is completable then it satisfies \hc\ by Lemma~\ref{hconlyif}. Conversely, suppose $P$ satisfies \hc. Then \HI{H} holds, which means that
\[ \sum_{\sigma = 1}^n \alpha(\sigma, H) \ge rn,\]
which can only be the case if $\alpha(\sigma, H) = r$ for each $\sigma \in \{1, \dots, n\}$. By Lemma~\ref{alphasigma} we have $\nu(\sigma) + n - s \ge r$ for each $\sigma \in \{1, \dots, n \}$, and so $P$ satisfies Ryser's Condition. Thus $P$ is completable by Theorem~\ref{ryser}. \end{proof}

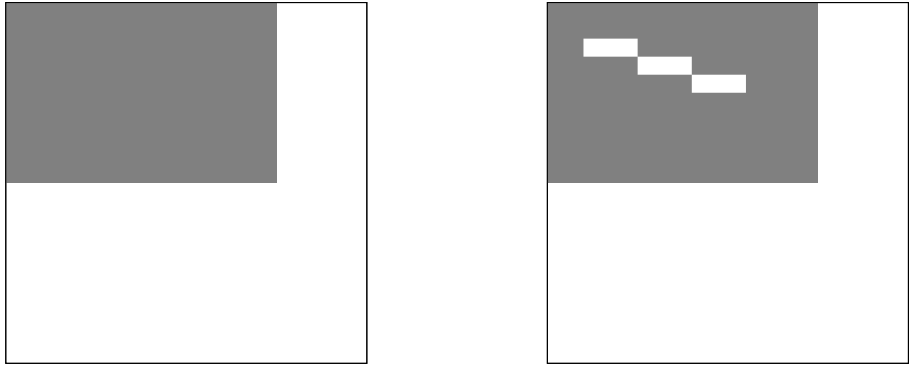
\begin{figure}
\begin{center} \begin{tikzpicture}
\begin{scope}[scale=1.2]
\draw[black] (0,-1) rectangle (4,3);
\fill[gray] (0,1) rectangle (3,3);
\draw[black] (0,-1) rectangle (4,3);
\begin{scope}[xshift=6cm]
\draw[black] (0,-1) rectangle (4,3);
\fill[gray] (0,1) rectangle (3,3);
\fill[white] (0.4,2.4) rectangle +(0.6,0.2)
(1,2.2) rectangle +(0.6,0.2)
(1.6,2) rectangle +(0.6,0.2);
\draw[black] (0,-1) rectangle (4,3);
\end{scope}
\end{scope}
\end{tikzpicture} \end{center}
\caption{The shapes of partial latin square considered in Theorems \ref{pls1} and \ref{pls5}.} \label{shapes12}
\end{figure}

The following lemma is a generalization of Lemma~\ref{alphasigma}, and is essential to the proof of Theorem~\ref{pls5}.

\begin{lemma} Let $P$ be a partial latin square of order $n$ whose filled cells are all in the upper left $r\times s$ rectangle $R$, for some $r,s \in \{1, \dots, n\}$, although at most one cell in each column inside the rectangle may be empty. Let $J$ be a subset of these empty cells. Let $H$ be the set of cells in the top $r$ rows, and for each $\sigma \in \{1, \dots, n\}$ let $\nu(\sigma)$ denote the number of times that $\sigma$ appears in $R$, and let $\rho(\sigma)$ be the number of rows in which there is an empty cell in $R-J$ that supports $\sigma$. Then
\[ \alpha(\sigma, H-J) = \min\{r, \ \nu(\sigma)+\rho(\sigma)+n-s\}. \]
\label{alphasigma2} \end{lemma}

\begin{proof} Fix a symbol $\sigma \in \{1, \dots, n\}$. Let $S_1$ be the set of cells in $H$ that contain $\sigma$, plus one empty cell from $R-J$ that supports $\sigma$ from each row that contains such a cell. $S_1$ is an independent set of size $\nu(\sigma) + \rho(\sigma)$. We can find an independent set $S_2$ of size $\min\{r-\nu(\sigma)-\rho(\sigma), \ n-s\}$ from the cells in the rightmost $n-s$ columns that are in the $r - \nu(\sigma) - \rho(\sigma)$ rows that have no cells in $S_1$. Then $S = S_1 \cup S_2$ is an independent set for $\sigma$ of size $\min\{r, \ \nu(\sigma) + \rho(\sigma) + n - s\}$. So $\alpha(\sigma, H-J) \ge \min\{r, \ \nu(\sigma)+ \rho(\sigma) + n - s\}$. Moreover, $\alpha(\sigma, H-J) \le r$ as there are $r$ rows in $H$; also $\alpha(\sigma, H-J) \le \nu(\sigma) + \rho(\sigma) + n - s$ as no independent set for $\sigma$ can use more than $\nu(\sigma) + \rho(\sigma) + n - s$ columns. So in fact we have $\alpha(\sigma, H-J) = \min\{r, \ \nu(\sigma) + \rho(\sigma) + n - s\}$. \end{proof}

We can now prove Theorem~\ref{pls5}.

\begin{proof}[Proof of Theorem~\ref{pls5}] If $P$ is completable then it satisfies \hc\ by Lemma~\ref{hconlyif}. Conversely, suppose $P$ satisfies \hc.

Let $H$ be the set of $rn$ cells in the first $r$ rows of $P$, and let $B$ be the set of empty cells contained within $R$. We shall give a procedure for finding a completion of $P$. The procedure consists of three steps, which we now outline.

\begin{enumerate}[{Step} 1.]
\item A partial latin square $Q_1$ is constructed from $P$, by filling \textit{some} of the cells of $B$, in such a way that each symbol appears at least $r + s - n$ times within $R$. This step will be shown to be possible because the \hil\ \HI{H-B'} holds for each subset $B' \subseteq B$. The main tool used to show this will be Theorem~\ref{maxflow} (the Max-flow Min-cut Theorem).

\item A partial latin square $Q_2$ is constructed from $P$ by filling \textit{all} the cells of $B$. This step will be performed entirely independently to Step 1; $Q_2$ will not depend on $Q_1$ in any way. For $Q_2$ there will be no conditions put on the number of times each symbol must appear in $R$. This step will be shown to be possible because the \hil\ \HI{B'} holds for each subset $B' \subseteq B$. The main tool used to show this will be Theorem~\ref{hall2} (Hall's Theorem).

\item A partial latin square $Q$ is constructed from $Q_1$ and $Q_2$, in which all the cells of $B$ are filled, and for which Ryser's Condition holds. The main tool used will be Theorem~\ref{dulmage} (the Dulmage-Mendelsohn Theorem). It follows from Theorem~\ref{ryser} (Ryser's Theorem) that $Q$ is completable.

\end{enumerate}

Note that we shall not need to make use of the fact that \textit{all} the Hall Inequalities are satisfied; it will suffice to make use of those inequalities that are of the form \HI{B'} or \HI{H-B'} for some $B' \subseteq B$.

\begin{center} Step 1. \end{center}

In this step, we shall describe a procedure for constructing a partial latin square $Q_1$ from $P$ by filling in \textit{some} of the cells of $B$ in such a way that each symbol $\sigma \in \{1, \dots, n\}$ appears at least $r + s - n$ times in $R$.

By Lemma~\ref{alphasigma2}, first with $J = B$, and second with $J = \emptyset$, we have
\[ \alpha(\sigma, H-B) = \min \{r, \ \nu(\sigma) + n - s\} \]
and
\[ \alpha(\sigma, H) = \min \{r, \ \nu(\sigma) + \rho(\sigma) + n - s\}, \] where $\rho(\sigma)$ is the number of rows in which $\sigma$ is supported by a cell of $B$. \HI{H} implies that for each $\sigma \in \{1, \dots, n\}$ we have $\alpha(\sigma, H) = r$, and therefore
\[ \nu(\sigma) + \rho(\sigma) + n - s \ge r. \]

For each symbol $\sigma \in \{1, \dots, n\}$ let $\mu(\sigma)$ be the least non-negative integer such that
\begin{equation} \nu(\sigma) + \mu(\sigma) \ge r + s - n. \label{nplusm} \end{equation}
Thus if $\sigma$ already appears at least $r + s - n$ times in $R$ we have $\mu(\sigma) = 0$; otherwise we have $\mu(\sigma) > 0$ and
\begin{equation} \nu(\sigma) + \mu(\sigma) = r + s - n. \label{nplusmeq} \end{equation}

Note that for each $\sigma \in \{1, \dots, n\}$ we have
\begin{equation} \mu(\sigma) \le \rho(\sigma). \label{mger} \end{equation}

Let
\[ u = \sum_{\sigma = 1}^n \mu(\sigma). \]
If we can fill in $u$ cells of $B$ using each symbol $\sigma \in \{1, \dots, n\}$ $\mu(\sigma)$ times, then in the resulting partial latin square, by (\ref{nplusm}), each symbol will appear at least $r + s - n$ times in $R$, and we will have the required partial latin square $Q_1$. We shall now show that this can always be done.

If a symbol $\sigma$ has $\mu(\sigma) = 0$ then no copies of this symbol need be placed in $B$. In fact, we only need be concerned with those symbols $\sigma$ for which $\mu(\sigma) > 0$; and for these symbols equation (\ref{nplusmeq}) holds.

Consider a directed graph $G$ with edge-capacities, that has vertex set $U \cup X \cup B \cup \{\alpha, \omega\}$ where
\[ U = \{ \sigma : \sigma \in \{1, \dots, n\} \text{\ and\ } \mu(\sigma) > 0\}, \]
\[ X = \{ (\sigma, w) : w \in \{1, \dots, r\} \text{\ and $\sigma$ is supported by at least one empty cell in row $w$} \} ,\] and $\alpha$ and $\omega$ are two additional vertices, which can be thought of as a ``source'' and a ``sink''. The source will have zero in-degree and the sink will have zero out-degree; we shall consider cuts in $G$ that separate $\alpha$ from $\omega$.

In $G$, $\alpha$ is joined to each vertex $\sigma \in U$ by an edge of capacity $\mu(\sigma)$, the direction being from $\alpha$ to $\sigma$. Each vertex $\sigma \in U$ is joined to all the vertices in $X$ of the form $(\sigma, w)$ for some $w \in \{1, \dots, r\}$ by edges of capacity 1, the direction being from $U$ to $X$. Each vertex $(\sigma, w) \in X$ is joined to all $b \in B$ where $b$ is a cell in row $w$ that supports $\sigma$, with edges of capacity $u$, the direction being from $X$ to $B$. Each vertex in $B$ is joined to $\omega$ by an edge of capacity 1, the direction being from $B$ to $\omega$. (See Figure~\ref{pls5fig}.)

\begin{figure}
\begin{center} \begin{tikzpicture}[scale=0.8,decoration={markings, mark=at position 0.66 with {\arrow{latex}}}]

\draw[black, postaction={decorate}] (0,2)--(2.7,3);
\draw[black, postaction={decorate}] (0,2)--(2.7,1);
\draw[black, postaction={decorate}] (0,2)--(2.7,2);
\draw[black, postaction={decorate}] (3.3,1)--(5.7,0.5);
\draw[black, postaction={decorate}] (3.3,2)--(5.7,2);
\draw[black, postaction={decorate}] (3.3,3)--(5.7,3.5);
\draw[black, postaction={decorate}] (6.3,0.5)--(8.7,0.8);
\draw[black, postaction={decorate}] (6.3,2)--(8.7,2);
\draw[black, postaction={decorate}] (6.3,3.5)--(8.7,3.2);

\draw[black, postaction={decorate}] (9.3,0.8)--(12,2);
\draw[black, postaction={decorate}] (9.3,2)--(12,2);
\draw[black, postaction={decorate}] (9.3,3.2)--(12,2);
\draw[draw=black, rounded corners=7] (2.7,0.5) rectangle (3.3,3.5);
\draw[draw=black, rounded corners=7] (5.7,0) rectangle (6.3,4);
\draw[draw=black, rounded corners=7] (8.7,0.3) rectangle (9.3,3.7);

\draw (0,2) node[left] {$\alpha$} (12,2) node[right] {$\omega$};
\fill (0,2) circle (2pt) (12,2) circle (2pt);
\draw (3,2) node {$U$};
\draw (6,2) node {$X$};
\draw (9,2) node {$B$};
\end{tikzpicture} \end{center}
\caption{The directed graph $G$ from Theorem~\ref{pls5}.} \label{pls5fig}
\end{figure}
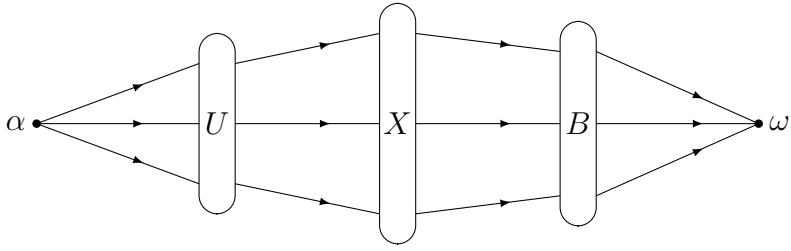

\vspace{5pt}

\begin{claim} It is possible to fill some of the cells of $B$ using each symbol $\sigma \in U$ $\mu(\sigma)$ times if and only if there is a $u$-flow in $G$ from $\alpha$ to $\omega$. \end{claim}

\begin{proof} Suppose there is such a partial filling of $B$. For each instance that a symbol $\sigma \in U$ is placed in cell $b \in B$ of row $w$, we can create a 1-flow in $G$ from $\alpha$ to $\omega$ by sending a flow of 1 from $\alpha$ to $\sigma \in U$ to $(\sigma, w) \in X$ to $b \in B$ to $\omega$. The sum of all these 1-flows gives a $u$-flow from $\alpha$ to $\omega$.

Conversely, suppose that there is a $u$-flow from $\alpha$ to $\omega$. By Theorem~\ref{maxflow} (the Max-flow Min-cut Theorem) there is such a flow that is integral. All the edges from $\alpha$ to $U$ must carry a flow equal to their capacity, and there must be $u$ paths from $U$ to $\omega$ each carrying 1-flows. These flows indicate how the cells of $B$ can be filled. The fact that the edges between $U$ and $X$ have capacity $1$ ensures that each symbol is placed at most once in each row. This proves Claim~1. \renewcommand{\qedsymbol}{} \end{proof}

By Claim~1, to show that $Q_1$ can be constructed, it will suffice to show that there is a $u$-flow in $G$ from $\alpha$ to $\omega$. So suppose, for a contradiction, that there does not exist such a $u$-flow in $G$. Then by Theorem~\ref{maxflow} (the Max-flow Min-cut Theorem) there must be a cut $T$ in $G$ of size less than $u$ that separates $\alpha$ from $\omega$. $T$ cannot contain any of the edges between $X$ and $B$ as each of these edges has capacity $u$. Also $T$ cannot contain \textit{all} the edges between $\alpha$ and $U$ (or it would have size $u$).

Let $U' \subseteq U$ be the set of vertices that $\alpha$ is joined to with an edge that is not in $T$. (See Figure~\ref{pls5fig2}.) For each $\sigma \in U'$ let $c(\sigma)$ be the number of edges joining $\sigma$ to $X$ that are in $T$. We may suppose that for all $\sigma \in U'$ we have $c(\sigma) < \mu(\sigma)$ as otherwise we can create a new cut $T'$ from $T$, where $\size{T'} \le \size{T}$, by removing the edges joining $\sigma$ to $X$ and adding the edge joining $\alpha$ to $\sigma$ instead.

Let $B' \subseteq B$ be the set of vertices in $B$ that are joined to vertices in $U'$ by paths containing no edges of $T$. Because the cut $T$ separates $\alpha$ from $\omega$ all the edges that join $B'$ to $\omega$ must be in $T$. Since we have assumed that $T$ has size less than $u$ we must have
\[ \sum_{\sigma \in U - U'} \mu(\sigma) + \sum_{\sigma \in U'} c(\sigma) + \size{B'} < u = \sum_{\sigma \in U} \mu(\sigma), \]
so that
\begin{equation} \sum_{\sigma \in U'} \left( \mu(\sigma) - c(\sigma) \right) > \size{B'}. \label{keyineq} \end{equation}

By Lemma~\ref{alphasigma2} we have for each $\sigma \in U$,
\begin{align*} \alpha(\sigma, H-B') &= \min\{r, \ \nu(\sigma) + \rho'(\sigma) + n - s\} \\
&\le \nu(\sigma) + \rho'(\sigma) + n - s,
\end{align*}
where $\rho'(\sigma)$ is the number of rows in which $\sigma$ is supported by a cell in $B - B'$. Then by (\ref{nplusmeq}) we have
\begin{equation}
\alpha(\sigma, H-B') \le r - \mu(\sigma) + \rho'(\sigma). \label{nurhodash}
\end{equation}

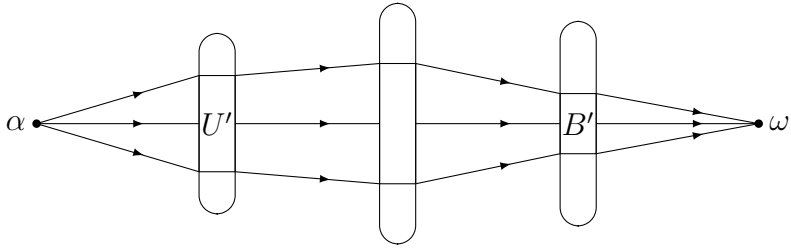
\begin{figure}
\begin{center} \begin{tikzpicture}[scale=0.8,decoration={markings, mark=at position 0.66 with {\arrow{latex}}}]



\draw[black, postaction={decorate}] (0,2)--(2.7,2.8);
\draw[black, postaction={decorate}] (0,2)--(2.7,1.2);
\draw[black, postaction={decorate}] (0,2)--(2.7,2);

\draw[black, postaction={decorate}] (3.3,1.2)--(5.7,1);
\draw[black, postaction={decorate}] (3.3,2)--(5.7,2);
\draw[black, postaction={decorate}] (3.3,2.8)--(5.7,3);

\draw[black, postaction={decorate}] (6.3,1)--(8.7,1.5);
\draw[black, postaction={decorate}] (6.3,2)--(8.7,2);
\draw[black, postaction={decorate}] (6.3,3)--(8.7,2.5);

\draw[black, postaction={decorate}] (9.3,1.5)--(12,2);
\draw[black, postaction={decorate}] (9.3,2)--(12,2);
\draw[black, postaction={decorate}] (9.3,2.5)--(12,2);

\draw[draw=black, rounded corners=7] (2.7,0.5) rectangle (3.3,3.5);
\draw[draw=black] (2.7,1.2) rectangle (3.3,2.8);
\draw[draw=black, rounded corners=7] (5.7,0) rectangle (6.3,4);
\draw[draw=black] (5.7,1) rectangle (6.3,3);
\draw[draw=black, rounded corners=7] (8.7,0.3) rectangle (9.3,3.7);
\draw[draw=black] (8.7,1.5) rectangle (9.3,2.5);

\draw (0,2) node[left] {$\alpha$} (12,2) node[right] {$\omega$};
\fill (0,2) circle (2pt) (12,2) circle (2pt);
\draw (3,2) node {$U'$};
\draw (9,2) node {$B'$};
\end{tikzpicture} \end{center}
\caption{The sets $U'$ and $B'$ from Theorem~\ref{pls5}.} \label{pls5fig2}
\end{figure}

Consider a row $w$ that contains at least one empty cell $b \in B - B'$ that supports $\sigma$. Each such row contributes $1$ to $\rho'(\sigma)$. The occurrence of cell $b \in B-B'$ in row $w$ supporting $\sigma$ corresponds to a unique path joining $\sigma \in U'$ to $b \in B-B'$ and passing through $(\sigma, w)$ in $X$. Such a path must contain one of the $c(\sigma)$ edges of $T$ joining $\sigma \in U'$ to $X$. Therefore for $\sigma \in U'$ we have
\[\rho'(\sigma) \le c(\sigma), \]
and therefore by (\ref{nurhodash}),
\[ \alpha(\sigma, H-B') \le r - \mu(\sigma) + c(\sigma). \]

So \begin{align*}
\sum_{\sigma = 1}^n \alpha(\sigma, H-B') &\le \sum_{\sigma \in U-U'} r + \sum_{\sigma \in U'} \left( r - \mu(\sigma) + c(\sigma) \right) \\
&\le rn - \sum_{\sigma \in U'} \left( \mu(\sigma) - c(\sigma) \right) \\
&< rn - \size{B'}, & \text{(by (\ref{keyineq}))}
\end{align*} which contradicts \HI{H-B'}.

Hence there must be a $u$-flow from $\alpha$ to $\omega$, and so by Claim~1 $Q_1$ can be constructed.

\begin{center} Step 2. \end{center}

In this step, we shall describe a procedure for constructing a partial latin square $Q_2$ from $P$ by filling \textit{all} the cells of $B$.

This can be done by filling the cells of $B$ row by row. We define the bipartite graphs $G_1, \dots, G_r$ as follows. For $w \in \{1, \dots, r\}$, $G_w$ is the graph with vertex set $S \cup B_w$ where $S = \{1, \dots, n\}$ is the set of symbols and $B_w$ is the subset of $B$ consisting of the members of $B$ that are in row $w$; where a cell $b \in B_w$ is joined to each symbol in $S$ that it supports. Note that some of the graphs $G_1, \dots, G_r$ may have no edges; in this case they can be disregarded. Note that the graphs are constructed in such a way that a matching in $G_w$ that covers $B_w$ corresponds to a filling of the cells $B_w$ where each cell is filled with a symbol that it supports and no two cells are filled with the same symbol.

\begin{claim} There is a matching in $G_w$ that covers $B_w$. \end{claim}

\begin{proof} Suppose not. Then by Theorem~\ref{hall2} (Hall's Theorem) there is a subset $B' \subseteq B_w$ with neighbour set $Y \in S$ such that $\size{Y} < \size{B'}$. But for any symbol $\sigma \in \{1, \dots, n\}$ we have $\alpha(\sigma, B') \le 1$, so \HI{B'} implies that the cells of $B'$ support at least $\size{B'}$ symbols, which contradicts $B'$ having fewer than $\size{B'}$ neighbours in $G_w$. This proves Claim 2. \renewcommand{\qedsymbol}{} \end{proof}

It follows that it is possible to fill all the cells of $B_w$ for each $w \in \{1, \dots, r\}$, and so in this way $Q_2$ can be constructed.

\begin{center} Step 3. \end{center}

In this step, we shall describe a procedure for constructing a partial latin square $Q_3$ from $Q_1$ and $Q_2$, in which all the cells of $B$ are filled, and each symbol appears at least $r + s - n$ times within $R$. We shall make further use of the bipartite graphs $G_1, \dots, G_r$ defined in Step 2.

For each row $w \in \{1, \dots, r\}$ that contains cells in $B$, the way that the cells of $B_w$ are filled in $Q_1$ and $Q_2$ give two matchings in $G_w$, $M_1$ and $M_2$. By Theorem~\ref{dulmage} (the Dulmage-Mendelsohn Theorem), there is a matching $M_w \subseteq M_1 \cup M_2$ that covers all the vertices in $B_w$ that are covered by $M_2$ (i.e{.} in fact, \textit{all} the vertices of $B_w$), and all the vertices in $S$ that are covered by $M_1$.

Then $Q$ can be created by taking $P$ and filling the cells of $B$ according to the matching $M_w$ for each row $w$. Since for each $w \in \{1, \dots, r\}$, $B_w$ contains all the symbols that it does in $Q_1$, each symbol $\sigma \in U$ appears at least $\mu(\sigma)$ times in $B$. It follows that in $Q$ each symbol appears at least $r + s - n$ times.

In $Q$, all the cells of $B$ are filled, and each symbol appears at least $r + s - n$ times. Thus $Q$ satisfies Ryser's Condition, and so by Theorem~\ref{ryser} it is possible to fill in the empty cells of $Q$ to create a latin square. This proves Theorem~\ref{pls5}. \end{proof}

\section{Complexity questions} \label{sec:complexity}

In this section we give the proof of Theorem~\ref{plsnph}. We need the following theorem of Kratochv\'il \cite{kratochvil}. A \textit{$(k$-in-$m)$-colouring} of an $m$-uniform hypergraph is a colouring of the vertices with red and blue such that each edge contains exactly $k$ red vertices and $m - k$ blue vertices.

\begin{theorem} For every $q\ge 3$, $m\ge 3$ and $1\le k\le m-1$, the problem of deciding $(k$-in-$m)$-colourability of $q$-regular $m$-uniform hypergraphs is NP-complete. \label{kratochvil} \end{theorem}

So in particular, the following problem is NP-complete.

\begin{problem} Let $H$ be a 4-uniform 4-regular hypergraph. Decide if $H$ is 2-in-4 colourable. \label{2-in-4} \end{problem}

A partial latin square of order $n$ is said to be \textit{L-shaped} if all the cells are filled except for those in the upper left $r\times s$ rectangle, for some $r,s \in \{1, \dots, n\}$. (This condition on the shape is opposite to that required for Theorems \ref{ryser} and \ref{pls1}.)

\begin{problem} Let $P$ be an L-shaped partial latin square. Decide if $P$ is completable. \label{lshapeprob} \end{problem}

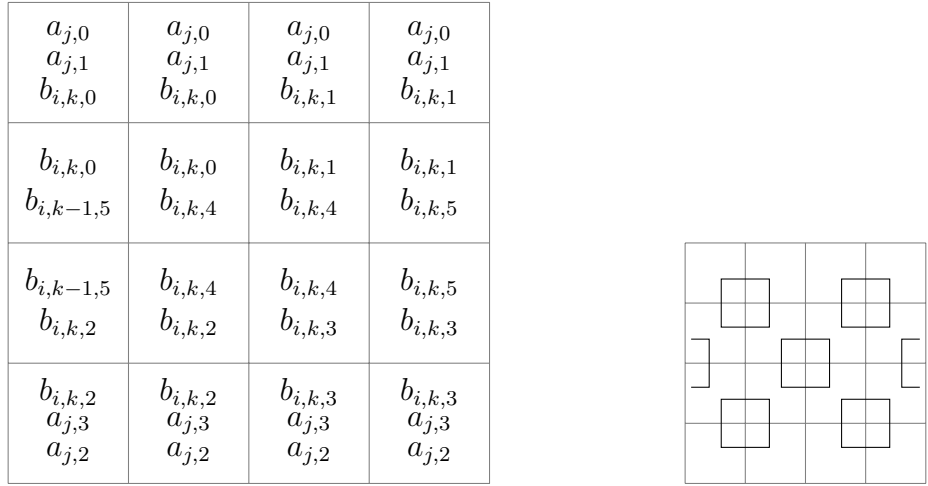
\begin{figure}
\begin{center} \begin{tikzpicture}
\begin{scope}[scale=0.8]
\draw[gray,step=2cm] (0,0) grid (8,8);
\draw (1,7.5) node {$a_{j,0}$} (1,7) node {$a_{j,1}$} (1,6.5) node {$b_{i,k,0}$}
(3,7.5) node {$a_{j,0}$} (3,7) node {$a_{j,1}$} (3,6.5) node {$b_{i,k,0}$}
(5,7.5) node {$a_{j,0}$} (5,7) node {$a_{j,1}$} (5,6.5) node {$b_{i,k,1}$}
(7,7.5) node {$a_{j,0}$} (7,7) node {$a_{j,1}$} (7,6.5) node {$b_{i,k,1}$}
(1,4.66) node {$b_{i,k-1,5}$} (1,5.33) node {$b_{i,k,0}$}
(3,4.66) node {$b_{i,k,4}$} (3,5.33) node {$b_{i,k,0}$}
(5,4.66) node {$b_{i,k,4}$} (5,5.33) node {$b_{i,k,1}$}
(7,4.66) node {$b_{i,k,5}$} (7,5.33) node {$b_{i,k,1}$}
(1,3.33) node {$b_{i,k-1,5}$} (1,2.66) node {$b_{i,k,2}$}
(3,3.33) node {$b_{i,k,4}$} (3,2.66) node {$b_{i,k,2}$}
(5,3.33) node {$b_{i,k,4}$} (5,2.66) node {$b_{i,k,3}$}
(7,3.33) node {$b_{i,k,5}$} (7,2.66) node {$b_{i,k,3}$}
(1,0.5) node {$a_{j,2}$} (1,1) node {$a_{j,3}$} (1,1.5) node {$b_{i,k,2}$}
(3,0.5) node {$a_{j,2}$} (3,1) node {$a_{j,3}$} (3,1.5) node {$b_{i,k,2}$}
(5,0.5) node {$a_{j,2}$} (5,1) node {$a_{j,3}$} (5,1.5) node {$b_{i,k,3}$}
(7,0.5) node {$a_{j,2}$} (7,1) node {$a_{j,3}$} (7,1.5) node {$b_{i,k,3}$};
\end{scope}
\begin{scope}[xshift=9cm, scale=0.8]
\draw[gray,step=1cm] (0,0) grid (4,4);
\draw[black] (0.6,0.6) rectangle +(0.8,0.8)
(2.6,0.6) rectangle +(0.8,0.8)
(0.6,2.6) rectangle +(0.8,0.8)
(2.6,2.6) rectangle +(0.8,0.8)
(1.6,1.6) rectangle +(0.8,0.8)
(0.1,1.6)--(0.4,1.6)--(0.4,2.4)--(0.1,2.4)
(3.9,1.6)--(3.6,1.6)--(3.6,2.4)--(3.9,2.4);
\end{scope}
\end{tikzpicture} \end{center}
\caption{The symbols in the $4 \times 4$ square of $M$ whose top-left cell is $(4i,4j)$ (left), and the positions of the symbols in $B$ indicated by rectangles (right).} \label{4x4} \end{figure}

We shall show that Problem~\ref{lshapeprob} is NP-complete, by giving a reduction from Problem~\ref{2-in-4} to Problem~\ref{lshapeprob}.\footnote{The reduction is a \textit{Karp reduction}, sometimes called a ``polynomial transformation''. All this terminology is discussed in \cite{gareyjohnson}.} The reduction has the property that its image is contained in the set of L-shaped partial latin squares that satisfy \hc.

\begin{lemma} Problem~\ref{lshapeprob} is NP-complete. Moreover, there is a reduction from Problem~\ref{2-in-4} to Problem~\ref{lshapeprob} that maps 4-uniform 4-regular hypergraphs to L-shaped partial latin squares that satisfy \hc. \label{plsred} \end{lemma}

An $r \times s$ \textit{latin rectangle} is an $r \times s$ array in which each cell is filled, and no symbol appears more than once in any row or column. A \textit{framework} is a tuple \[R = (r, s, t, R_1, \dots, R_r, C_1, \dots, C_s)\] where $r$, $s$ and $t$ are positive integers, and $R_1, \dots, R_r, C_1, \dots, C_s$ are subsets of $\{1, \dots, t\}$. The sets $R_1, \dots, R_r$ are called the \textit{row lists}, and the sets $C_1, \dots, C_s$ are called the \textit{column lists}.\footnote{Our frameworks are a special case of the ``latin frameworks'' from \cite{colbourn, eastonparker} and the ``patterned holes'' from \cite{lindnerrodger}.}

A \textit{latinization} of a framework $R = (r, s, t, R_1, \dots, R_r, C_1, \dots, C_s)$ is an $r \times s$ latin rectangle where each symbol that appears in row $i$ belongs to $R_i$, for each $1 \le i \le r$, and each symbol that appears in column $j$ belongs to $C_j$, for each $1 \le j \le s$. 

An L-shaped partial latin square $P$ is said to \textit{realize} $R$ (or be a \textit{realization} of $R$), if $R_i$ is the set of symbols missing from row $i$, for each $1 \le i \le r$, and $C_j$ is the set of symbols missing from column $j$, for each $1 \le j \le s$. Thus if $P$ realizes $R$, the upper left $r \times s$ rectangle of any completion of $P$ is a latinization of $R$, and \textit{vice-versa}.

If $P$ is an L-shaped partial latin square, where the upper left $r \times s$ rectangle is empty, then we can create a framework $R = (r, s, t, R_1, \dots, R_r, C_1, \dots, C_s)$, for which $P$ is a realization, in the following natural way. For each $1 \le i \le r$, let $R_i$ be the set of symbols missing from row $i$, and for each $1 \le j \le s$, let $C_j$ be the set of symbols missing from column $j$. Such a framework will have the following properties.

\begin{enumerate}[(i)]
 \item $\size{R_i} = s$ for each $1 \le i \le r$.
 \item $\size{C_j} = r$ for each $1 \le j \le s$.
 \item Each symbol appears the same number of times in the row lists as it does in the column lists.
\end{enumerate}

If a framework satisfies conditions (i), (ii) and (iii), it is said to be \textit{balanced}. It turns out that given any balanced framework $R$, there is an L-shaped partial latin square $P$ that realizes $R$.

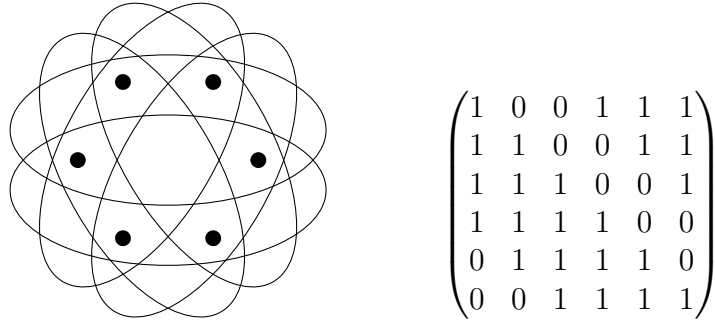
\begin{figure} \begin{center}
\begin{tikzpicture}
\draw (0,1.5) node {$\begin{pmatrix}
1&0&0&1&1&1\\
1&1&0&0&1&1\\
1&1&1&0&0&1\\
1&1&1&1&0&0\\
0&1&1&1&1&0\\
0&0&1&1&1&1
\end{pmatrix}$};
\begin{scope}[xshift=-5.5cm,yshift=2.1cm]
\draw[rotate=0]  (0,0.4) ellipse (2.1 and 1);
\draw[rotate=60]  (0,0.4) ellipse (2.1 and 1);
\draw[rotate=120] (0,0.4) ellipse (2.1 and 1);
\draw[rotate=180] (0,0.4) ellipse (2.1 and 1);
\draw[rotate=240]  (0,0.4) ellipse (2.1 and 1);
\draw[rotate=300]  (0,0.4) ellipse (2.1 and 1);
\fill[rotate=30] (0,1.2) circle(3pt);
\fill[rotate=90] (0,1.2) circle(3pt);
\fill[rotate=150] (0,1.2) circle(3pt);
\fill[rotate=210] (0,1.2) circle(3pt);
\fill[rotate=270] (0,1.2) circle(3pt);
\fill[rotate=330] (0,1.2) circle(3pt);
\end{scope}
\end{tikzpicture} \caption{An example of a 4-uniform 4-regular hypergraph $H$, where the edges are drawn as ellipses (left), and its incidence matrix $D$, under a suitable labelling of the vertices and edges (right).} \label{rh} \end{center} \end{figure}

\begin{figure} \begin{center}
\begin{tikzpicture} \input{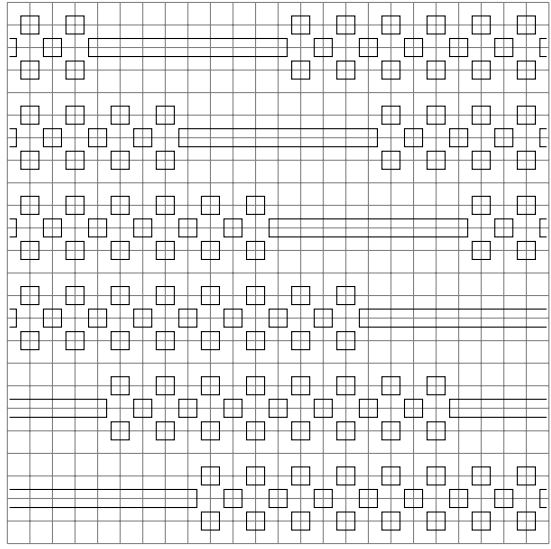}
\end{tikzpicture} \caption{The position of the symbols from $B$ in $M$.} \label{rh2}
\end{center} \end{figure}

\begin{theorem} Let $R = (r, s, t, R_1, \dots, R_r, C_1, \dots, C_s)$ be a balanced framework. For any $n \ge \max \{t, r+s\}$, there is a partial latin square $P$ of order $n$ that realizes $R$, and such a partial latin square $P$ can be found in polynomial time. \label{fr1} \end{theorem}

\begin{proof} Fix an $n$ such that $n \ge \max \{t, r+s\}$. We shall give a procedure, consisting of two steps, for constructing an L-shaped partial latin square $P$ of order $n$ that realizes $R$. Initially, suppose $P$ is a partial latin square of order $n$ with all cells empty. In the first step, we fill the cells in the rightmost $n-s$ columns of the top $r$ rows of $P$, and in the second step, we fill the cells in the bottom $n-r$ rows.

Step 1 is as follows. Consider the bipartite graph $G_1$ with bipartition $(A, B)$, where the vertices in $A$ are labelled $a_1, \dots, a_n$ and correspond to the symbols $1, \dots, n$; the vertices in $B$ are labelled $b_1, \dots, b_r$ and correspond to rows $1, \dots, r$ of $P$; and where there is an edge between $a_i$ and $b_j$ if and only if symbol $i$ is not in $R_j$, for each $1 \le i \le n$ and $1 \le j \le r$.

Thus $a_i$ has degree $r - \nu(i)$, where $\nu(i)$ is the number of times that the symbol $i$ occurs in the row lists, for each $1 \le i \le n$; and $b_j$ has degree $n - s$, as there are $s$ symbols in the row list $R_j$, for each $1 \le j \le r$. The maximum degree $\Delta(G_1)$ is $n - s$ for otherwise there would be some $i$ for which $r - \nu(i) > n - s$, which contradicts the assumption that $n \ge r + s$.

By Theorem~\ref{konig} (K\"onig's Theorem), $G_1$ has a proper edge-colouring with the colours $1, \dots, n-s$. So the cells in the rightmost $n-s$ columns of the top $r$ rows of $P$ can be filled according to this edge-colouring: symbol $i$ is placed in cell $(j, s + k)$ if and only if $a_i$ is joined to $b_j$ with an edge of colour $k$, for each $1 \le i \le n$, $1 \le j \le r$, $1 \le k \le n - s$. At this point the number of times the symbol $i$ either appears in the top $r$ rows of $P$ or in one of the row lists $R_1, \dots, R_r$ is $r$, for each $1 \le i \le n$. And because the framework $R$ is balanced, the number of times $i$ appears in the column lists $C_1, \dots, C_s$ or in the top $r$ rows of $P$ is also $r$, for each $1 \le i \le n$.

Step 2 is as follows. Consider the bipartite graph $G_2$ with bipartition $(C, D)$, where the vertices in $C$ are labelled $c_1, \dots, c_n$ and correspond to the $n$ columns of $P$, the vertices in $D$ are labelled $d_1, \dots, d_n$ and correspond to the symbols $1, \dots, n$, and where there is an edge between $d_i$ and $c_j$ if and only if $i$ is not in $C_j$ (when $1 \le j \le s$) or $i$ does not appear in the top $r$ cells of column $j$ of $P$ (when $s < j \le n$) for each $1 \le i \le n$ and $1 \le j \le n$. Each vertex in $G_2$ has degree $n-r$, so by Theorem~\ref{konig} (K\"onig's Theorem) we can give $G_2$ a proper edge-colouring using the colours $1, \dots, n - r$. We can now fill the bottom $n - r$ rows of $P$ according to this edge-colouring: symbol $i$ in placed in cell $(r + k, j)$ if and only if $d_i$ is joined to $c_j$ with an edge of colour $k$, for each $1 \le i \le n$, $1 \le j \le n$, $1 \le k \le n - r$.

This procedure creates an L-shaped partial latin square $P$, where the symbols missing from rows $1, \dots, r$ are those in the row lists $R_1, \dots, R_r$ and the symbols missing from columns $1, \dots, s$ are those in the column lists $C_1, \dots, C_s$. In other words, $P$ is a realization of $R$. The procedure involves edge-colouring two graphs, each with at most $2n$ vertices, which can be done in $O(n^3)$ time (see e.g. \cite[Chapter 20]{schrijver}). Since $n$ is polynomial in the size of the framework, the procedure can be performed in polynomial time. \end{proof}

Given a framework $R$, the \textit{admissible symbol array} $M = M(R)$ is an $r \times s$ array in which $M_{ij} = R_i \cap C_j$ for each $1 \le i \le r$ and $1 \le j \le s$. Thus a latinization of $R$ can be described as an $r \times s$ latin rectangle $L$ for which $L_{ij} \in M_{ij}$ for each $1 \le i \le r$ and $1 \le j \le s$. We are now ready to prove Lemma~\ref{plsred}.

\begin{proof}[Proof of Lemma~\ref{plsred}]
Let $H$ be a 4-uniform 4-regular hypergraph of order $u$. We shall give a procedure for constructing an L-shaped partial latin square $Q$ such that $Q$ is completable if and only if $H$ is 2-in-4 colourable. $Q$ will be an L-shaped partial latin square of order $n = 4u^2 + 12u$, in which the upper left $4u \times 4u$ square is empty, and the rest of the cells are filled. In addition, $Q$ will satisfy \hc.

Note that since $H$ is 4-uniform and 4-regular, $H$ has the same number of vertices as edges. So $H$ has $u$ vertices and $u$ edges, which we assume are labelled $v_0, \dots, v_{u-1}$ and $e_0, \dots, e_{u-1}$ respectively. We shall describe the construction from $H$ of a balanced framework
\[R = R(H) = (4u, 4u, n, R_1, \dots, R_{4u}, C_1, \dots, C_{4u}), \]
with the property that $R$ is latinizable if and only if $H$ is 2-in-4 colourable. Instead of describing $R$ directly, we describe its admissible symbol array $M$.

$M$ is a $4u \times 4u$ array whose entries are subsets of $\{1, \dots, n\}$. For convenience we consider the set $\{1, \dots, n\}$ to be the union of three sets $A$, $B$ and $C$, whose members are labelled as follows. $A$ consists of $a_{j,k}$ for each $0 \le j \le u - 1$ and $0 \le k \le 3$; $B$ consists of $b_{i,j,k}$ for each $0 \le i \le u - 1$, $0 \le j \le 3$ and $0 \le k \le 5$; and $C$ consists of $c_{i,j,k}$ for each $0 \le i \le u - 1$, $0 \le j \le u - 1$, $0 \le k \le 3$ for which $v_i$ is not incident with $e_j$.
Thus $\size{A} = 4u$, $\size{B} = 24u$, $\size{C} = 4u(u-4)$ and $n = \size{A} + \size{B} + \size{C}$.

We can consider the cells of $M$ as consisting of a $u \times u$ grid of $4 \times 4$ squares. In the construction of $M$, the entries in the $4 \times 4$ square whose top-left cell is $(4i, 4j)$ are determined by whether vertex $v_i$ is incident with the edge $e_j$.

If $v_i$ is incident with $e_j$, and $e_j$ is the $(k+1)$st edge in the ordering $e_0, \dots, e_{u-1}$ that is incident with $v_i$, then the cells of the $4 \times 4$ square whose top-left cell is $(4i, 4j)$ contain some symbols from $A$ and some from $B$. Each cell in the top row contains the symbols $a_{j,0}$ and $a_{j,1}$, and each cell in the bottom row contains the symbols $a_{j,2}$ and $a_{j,3}$. The cells contain the symbols $b_{i,k,l}$, for each $l \in \{0, \dots, 5\}$, and $b_{i,k-1,5}$ (where subtraction is taken modulo 4), in the manner shown in Figure~\ref{4x4}. The right-hand diagram gives a simplified picture, where the positions of the symbols in $B$ are indicated by rectangles; for each symbol in $B$, there is a rectangle whose corners are located in the cells that contain it.

If $v_i$ is not incident with $e_j$, then the cells of the $4 \times 4$ square whose top-left cell is $(4i, 4j)$ contain some symbols from $C$; namely the symbols $c_{i,j,0}$, $c_{i,j,1}$, $c_{i,j,2}$ and $c_{i,j,3}$.

Figure~\ref{rh} gives an example of the a 4-uniform 4-regular hypergraph $H$ and its incidence matrix. Figure~\ref{rh2} gives a simplified illustration of the array $M$, constructed from $H$, where the positions of the symbols from $B$ are indicated by rectangles.

It is easy to verify that the cells in each row of $M$ contain $4u$ symbols, and the cells in each column of $M$ contain $4u$ symbols. Each symbol in $A$ and $C$ belongs to 16 cells, in 4 rows and 4 columns. Each symbol in $B$ belongs to 4 cells, in 2 rows and 2 columns. It is not hard to construct a balanced framework $R$ that has $M$ as its admissible symbol array.

\begin{figure} \begin{center}
\begin{tikzpicture}
\draw[gray,step=0.600cm] (0,0) grid (7.200,2.400);
\draw[black] (0.360,0.360) rectangle +(0.480,0.480); \fill[black] (0.360,0.840) circle (2pt) (0.840,0.360) circle (2pt); \draw[black] (1.560,0.360) rectangle +(0.480,0.480); \fill[black] (1.560,0.840) circle (2pt) (2.040,0.360) circle (2pt); \draw[black] (0.360,1.560) rectangle +(0.480,0.480); \fill[black] (0.360,2.040) circle (2pt) (0.840,1.560) circle (2pt); \draw[black] (1.560,1.560) rectangle +(0.480,0.480); \fill[black] (1.560,2.040) circle (2pt) (2.040,1.560) circle (2pt); \draw[black] (0.960,0.960) rectangle +(0.480,0.480); \fill[black] (0.960,0.960) circle (2pt) (1.440,1.440) circle (2pt); \begin{scope} \clip (0.060,0.480) rectangle (0.360,1.920); \draw[black] (-0.240,0.960) rectangle +(0.480,0.480); \fill[black] (-0.240,0.960) circle (2pt) (0.240,1.440) circle (2pt); \end{scope} \draw[black] (2.160,0.960) rectangle +(0.480,0.480); \fill[black] (2.160,0.960) circle (2pt) (2.640,1.440) circle (2pt); \draw[black] (2.760,0.360) rectangle +(0.480,0.480); \fill[black] (2.760,0.840) circle (2pt) (3.240,0.360) circle (2pt); \draw[black] (3.960,0.360) rectangle +(0.480,0.480); \fill[black] (3.960,0.840) circle (2pt) (4.440,0.360) circle (2pt); \draw[black] (2.760,1.560) rectangle +(0.480,0.480); \fill[black] (2.760,2.040) circle (2pt) (3.240,1.560) circle (2pt); \draw[black] (3.960,1.560) rectangle +(0.480,0.480); \fill[black] (3.960,2.040) circle (2pt) (4.440,1.560) circle (2pt); \draw[black] (3.360,0.960) rectangle +(0.480,0.480); \fill[black] (3.360,0.960) circle (2pt) (3.840,1.440) circle (2pt); \draw[black] (4.560,0.960) rectangle +(0.480,0.480); \fill[black] (4.560,0.960) circle (2pt) (5.040,1.440) circle (2pt); \draw[black] (5.160,0.360) rectangle +(0.480,0.480); \fill[black] (5.160,0.840) circle (2pt) (5.640,0.360) circle (2pt); \draw[black] (6.360,0.360) rectangle +(0.480,0.480); \fill[black] (6.360,0.840) circle (2pt) (6.840,0.360) circle (2pt); \draw[black] (5.160,1.560) rectangle +(0.480,0.480); \fill[black] (5.160,2.040) circle (2pt) (5.640,1.560) circle (2pt); \draw[black] (6.360,1.560) rectangle +(0.480,0.480); \fill[black] (6.360,2.040) circle (2pt) (6.840,1.560) circle (2pt); \draw[black] (5.760,0.960) rectangle +(0.480,0.480); \fill[black] (5.760,0.960) circle (2pt) (6.240,1.440) circle (2pt); \begin{scope} \clip (6.840,0.480) rectangle (7.140,1.920); \draw[black] (6.960,0.960) rectangle +(0.480,0.480); \fill[black] (6.960,0.960) circle (2pt) (7.440,1.440) circle (2pt); \end{scope} 
\begin{scope}[yshift=3cm, xshift=7.2cm, xscale=-1]
\draw[gray,step=0.600cm] (0,0) grid (7.200,2.400);
\draw[black] (0.360,0.360) rectangle +(0.480,0.480); \fill[black] (0.360,0.840) circle (2pt) (0.840,0.360) circle (2pt); \draw[black] (1.560,0.360) rectangle +(0.480,0.480); \fill[black] (1.560,0.840) circle (2pt) (2.040,0.360) circle (2pt); \draw[black] (0.360,1.560) rectangle +(0.480,0.480); \fill[black] (0.360,2.040) circle (2pt) (0.840,1.560) circle (2pt); \draw[black] (1.560,1.560) rectangle +(0.480,0.480); \fill[black] (1.560,2.040) circle (2pt) (2.040,1.560) circle (2pt); \draw[black] (0.960,0.960) rectangle +(0.480,0.480); \fill[black] (0.960,0.960) circle (2pt) (1.440,1.440) circle (2pt); \begin{scope} \clip (0.060,0.480) rectangle (0.360,1.920); \draw[black] (-0.240,0.960) rectangle +(0.480,0.480); \fill[black] (-0.240,0.960) circle (2pt) (0.240,1.440) circle (2pt); \end{scope} \draw[black] (2.160,0.960) rectangle +(0.480,0.480); \fill[black] (2.160,0.960) circle (2pt) (2.640,1.440) circle (2pt); \draw[black] (2.760,0.360) rectangle +(0.480,0.480); \fill[black] (2.760,0.840) circle (2pt) (3.240,0.360) circle (2pt); \draw[black] (3.960,0.360) rectangle +(0.480,0.480); \fill[black] (3.960,0.840) circle (2pt) (4.440,0.360) circle (2pt); \draw[black] (2.760,1.560) rectangle +(0.480,0.480); \fill[black] (2.760,2.040) circle (2pt) (3.240,1.560) circle (2pt); \draw[black] (3.960,1.560) rectangle +(0.480,0.480); \fill[black] (3.960,2.040) circle (2pt) (4.440,1.560) circle (2pt); \draw[black] (3.360,0.960) rectangle +(0.480,0.480); \fill[black] (3.360,0.960) circle (2pt) (3.840,1.440) circle (2pt); \draw[black] (4.560,0.960) rectangle +(0.480,0.480); \fill[black] (4.560,0.960) circle (2pt) (5.040,1.440) circle (2pt); \draw[black] (5.160,0.360) rectangle +(0.480,0.480); \fill[black] (5.160,0.840) circle (2pt) (5.640,0.360) circle (2pt); \draw[black] (6.360,0.360) rectangle +(0.480,0.480); \fill[black] (6.360,0.840) circle (2pt) (6.840,0.360) circle (2pt); \draw[black] (5.160,1.560) rectangle +(0.480,0.480); \fill[black] (5.160,2.040) circle (2pt) (5.640,1.560) circle (2pt); \draw[black] (6.360,1.560) rectangle +(0.480,0.480); \fill[black] (6.360,2.040) circle (2pt) (6.840,1.560) circle (2pt); \draw[black] (5.760,0.960) rectangle +(0.480,0.480); \fill[black] (5.760,0.960) circle (2pt) (6.240,1.440) circle (2pt); \begin{scope} \clip (6.840,0.480) rectangle (7.140,1.920); \draw[black] (6.960,0.960) rectangle +(0.480,0.480); \fill[black] (6.960,0.960) circle (2pt) (7.440,1.440) circle (2pt); \end{scope}  
\end{scope}
\end{tikzpicture} \caption{Symbols in ``red'' position (top) and ``blue'' position (bottom).} \label{rhslope}
\end{center} \end{figure}
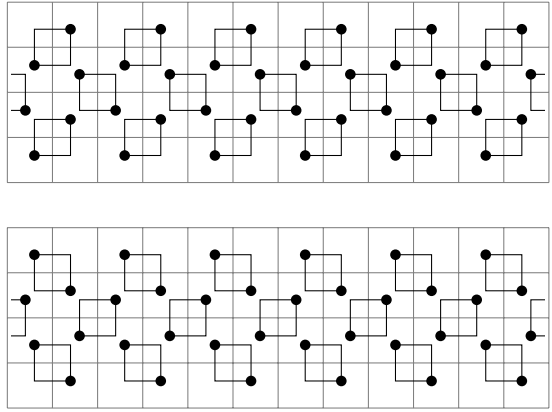

\begin{claimx} $R$ can be latinized if and only if $H$ is 2-in-4 colourable. \end{claimx}

\begin{proof} First suppose that $R$ is latinizable, and that $L$ is a latinization of $R$. Consider the symbols $b_{i,k,l}$, where $0 \le i \le u - 1$ is fixed, and $k$ and $l$ range over $\{0, \dots, 3\}$ and $\{0, \dots, 5\}$ respectively. These symbols are naturally associated with the vertex $v_i$. The first point to note is that the position of just one of these symbols in $L$ determines the position of all of them. Moreover, there are just two different ways in which they can be placed, called ``red position'' and ``blue position''. These are illustrated in Figure~\ref{rhslope}, where the dots indicate the positions of symbols from $B$.

We can colour the vertices of $H$ red or blue according to whether their associated symbols in $B$ are in red position or blue position. We claim that this in fact gives a 2-in-4 colouring of $H$. Consider the symbols $a_{j,k}$, where $0 \le j \le u - 1$ is fixed, and $k$ ranges over $\{0, \dots, 3\}$. Each of these symbols must appear four times in the columns $4j, \dots, 4j+3$ of $L$. But this can only happen if amongst the symbols in $B$ corresponding to the vertices incident with $e_j$, two are in red position, and two are in blue position. An example of how they might be placed is illustrated in Figure~\ref{rhcolumn}. Hence from $L$ we can create a 2-in-4 colouring of $H$, and so $H$ is 2-in-4 colorable.

\begin{figure} \begin{center} \begin{tikzpicture}
\tikzstyle{every node}=[font=\small]
\draw[black] (0.360,7.560) rectangle +(0.480,0.480); \fill[black] (0.360,7.560) circle (2pt) (0.840,8.040) circle (2pt); \draw[black] (1.560,7.560) rectangle +(0.480,0.480); \fill[black] (1.560,7.560) circle (2pt) (2.040,8.040) circle (2pt); \draw[black] (0.360,8.760) rectangle +(0.480,0.480); \fill[black] (0.360,8.760) circle (2pt) (0.840,9.240) circle (2pt); \draw[black] (1.560,8.760) rectangle +(0.480,0.480); \fill[black] (1.560,8.760) circle (2pt) (2.040,9.240) circle (2pt); \draw[black] (0.960,8.160) rectangle +(0.480,0.480); \fill[black] (0.960,8.640) circle (2pt) (1.440,8.160) circle (2pt); \begin{scope} \clip (0.060,7.680) rectangle (0.360,9.120); \draw[black] (-0.240,8.160) rectangle +(0.480,0.480); \fill[black] (-0.240,8.640) circle (2pt) (0.240,8.160) circle (2pt); \end{scope} \begin{scope} \clip (2.040,7.680) rectangle (2.340,9.120); \draw[black] (2.160,8.160) rectangle +(0.480,0.480); \fill[black] (2.160,8.640) circle (2pt) (2.640,8.160) circle (2pt); \end{scope} \draw[black] (0.360,5.160) rectangle +(0.480,0.480); \fill[black] (0.360,5.640) circle (2pt) (0.840,5.160) circle (2pt); \draw[black] (1.560,5.160) rectangle +(0.480,0.480); \fill[black] (1.560,5.640) circle (2pt) (2.040,5.160) circle (2pt); \draw[black] (0.360,6.360) rectangle +(0.480,0.480); \fill[black] (0.360,6.840) circle (2pt) (0.840,6.360) circle (2pt); \draw[black] (1.560,6.360) rectangle +(0.480,0.480); \fill[black] (1.560,6.840) circle (2pt) (2.040,6.360) circle (2pt); \draw[black] (0.960,5.760) rectangle +(0.480,0.480); \fill[black] (0.960,5.760) circle (2pt) (1.440,6.240) circle (2pt); \begin{scope} \clip (0.060,5.280) rectangle (0.360,6.720); \draw[black] (-0.240,5.760) rectangle +(0.480,0.480); \fill[black] (-0.240,5.760) circle (2pt) (0.240,6.240) circle (2pt); \end{scope} \begin{scope} \clip (2.040,5.280) rectangle (2.340,6.720); \draw[black] (2.160,5.760) rectangle +(0.480,0.480); \fill[black] (2.160,5.760) circle (2pt) (2.640,6.240) circle (2pt); \end{scope} \draw[black] (0.360,2.760) rectangle +(0.480,0.480); \fill[black] (0.360,2.760) circle (2pt) (0.840,3.240) circle (2pt); \draw[black] (1.560,2.760) rectangle +(0.480,0.480); \fill[black] (1.560,2.760) circle (2pt) (2.040,3.240) circle (2pt); \draw[black] (0.360,3.960) rectangle +(0.480,0.480); \fill[black] (0.360,3.960) circle (2pt) (0.840,4.440) circle (2pt); \draw[black] (1.560,3.960) rectangle +(0.480,0.480); \fill[black] (1.560,3.960) circle (2pt) (2.040,4.440) circle (2pt); \draw[black] (0.960,3.360) rectangle +(0.480,0.480); \fill[black] (0.960,3.840) circle (2pt) (1.440,3.360) circle (2pt); \begin{scope} \clip (0.060,2.880) rectangle (0.360,4.320); \draw[black] (-0.240,3.360) rectangle +(0.480,0.480); \fill[black] (-0.240,3.840) circle (2pt) (0.240,3.360) circle (2pt); \end{scope} \begin{scope} \clip (2.040,2.880) rectangle (2.340,4.320); \draw[black] (2.160,3.360) rectangle +(0.480,0.480); \fill[black] (2.160,3.840) circle (2pt) (2.640,3.360) circle (2pt); \end{scope} \draw[black] (0.360,0.360) rectangle +(0.480,0.480); \fill[black] (0.360,0.840) circle (2pt) (0.840,0.360) circle (2pt); \draw[black] (1.560,0.360) rectangle +(0.480,0.480); \fill[black] (1.560,0.840) circle (2pt) (2.040,0.360) circle (2pt); \draw[black] (0.360,1.560) rectangle +(0.480,0.480); \fill[black] (0.360,2.040) circle (2pt) (0.840,1.560) circle (2pt); \draw[black] (1.560,1.560) rectangle +(0.480,0.480); \fill[black] (1.560,2.040) circle (2pt) (2.040,1.560) circle (2pt); \draw[black] (0.960,0.960) rectangle +(0.480,0.480); \fill[black] (0.960,0.960) circle (2pt) (1.440,1.440) circle (2pt); \begin{scope} \clip (0.060,0.480) rectangle (0.360,1.920); \draw[black] (-0.240,0.960) rectangle +(0.480,0.480); \fill[black] (-0.240,0.960) circle (2pt) (0.240,1.440) circle (2pt); \end{scope} \begin{scope} \clip (2.040,0.480) rectangle (2.340,1.920); \draw[black] (2.160,0.960) rectangle +(0.480,0.480); \fill[black] (2.160,0.960) circle (2pt) (2.640,1.440) circle (2pt); \end{scope} 
\fill[gray!20] (0,0) rectangle +(0.6,0.6)
(0,4.2) rectangle +(0.6,0.6)
(0,4.8) rectangle +(0.6,0.6)
(0,9) rectangle +(0.6,0.6)
(1.2,0) rectangle +(0.6,0.6)
(1.2,4.2) rectangle +(0.6,0.6)
(1.2,4.8) rectangle +(0.6,0.6)
(1.2,9) rectangle +(0.6,0.6)
(0.6,1.8) rectangle +(0.6,0.6)
(0.6,2.4) rectangle +(0.6,0.6)
(0.6,6.6) rectangle +(0.6,0.6)
(0.6,7.2) rectangle +(0.6,0.6)
(1.8,1.8) rectangle +(0.6,0.6)
(1.8,2.4) rectangle +(0.6,0.6)
(1.8,6.6) rectangle +(0.6,0.6)
(1.8,7.2) rectangle +(0.6,0.6);
\draw[gray,step=0.600cm] (0,0) grid (2.400,9.600);
\path (0.3,9.3) node {$a_{j,0}$}
(0.9,6.9) node {$a_{j,0}$}
(1.5,4.5) node {$a_{j,0}$}
(2.1,2.1) node {$a_{j,0}$}

(1.5,9.3) node {$a_{j,1}$}
(2.1,6.9) node {$a_{j,1}$}
(0.3,4.5) node {$a_{j,1}$}
(0.9,2.1) node {$a_{j,1}$}

(0.9,7.5) node {$a_{j,2}$}
(0.3,5.1) node {$a_{j,2}$}
(2.1,2.7) node {$a_{j,2}$}
(1.5,0.3) node {$a_{j,2}$}

(2.1,7.5) node {$a_{j,3}$}
(1.5,5.1) node {$a_{j,3}$}
(0.9,2.7) node {$a_{j,3}$}
(0.3,0.3) node {$a_{j,3}$};

\end{tikzpicture} \caption{A possible arrangement of the symbols $a_{j,k}$ for $0 \le k \le 3$.} \label{rhcolumn}
\end{center} \end{figure}
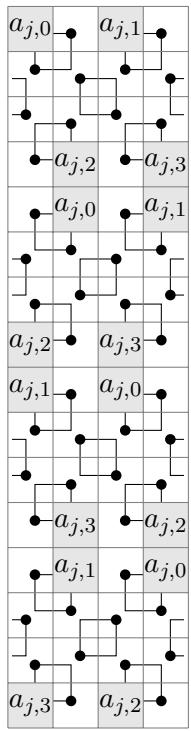

Conversely, suppose $H$ is 2-in-4 colourable. Then we can choose a 2-in-4 colouring of $H$, and place the symbols of $B$ in red or blue position according to whether their associated vertices in $H$ are red or blue. This leaves empty cells in which the symbols $a_{j,k}$ for all $0 \le j \le u - 1$ and $0 \le k \le 3$ can be placed. The symbols of $C$ can be placed without any difficulty, as the symbols $c_{i,j,0}$, $c_{i,j,1}$, $c_{i,j,2}$ and $c_{i,j,3}$ for a fixed $0 \le i \le u - 1$ and $0 \le j \le u - 1$ can only be placed in the $4 \times 4$ square whose top-left cell is $(4i, 4j)$. We can fill this square with any latin square on these four symbols. Hence $R$ is latinizable. This proves claim 1. \renewcommand{\qedsymbol}{} \end{proof}

We have described the construction of a balanced framework $R$ that is latinizable if and only if $H$ is 2-in-4 colourable. We can now use the procedure given in Theorem~\ref{fr1} to find an L-shaped partial latin square $Q$ of order $n = 4u^2 + 12u$ that realizes $R$. Thus $Q$ is completable if and only if $H$ is 2-in-4 colourable.

It remains to show that $Q$ satisfies \hc. Due to the way in which we constructed $R$, and subsequently $Q$, each empty cell supports either: (i) two symbols from $B$, (ii) one symbol from $B$ and two symbols from $A$, or (iii) four symbols from $C$. Each symbol from $B$ is missing from two rows and columns of $Q$, and each symbol from $A$ and $C$ is missing from four rows and columns of $Q$. Hence the conditions of Theorem~\ref{atleast1t} are satisfied, and $Q$ satisfies \hc.

The computation of the framework $R$ from $H$ can clearly be performed in polynomial time, and as shown in the proof of Theorem~\ref{fr1}, $Q$ can be computed from $R$ in polynomial time.
\end{proof}

A Turing machine equipped with an oracle for deciding Problem~\ref{plshcprob} could decide Problem~\ref{2-in-4} in polynomial time by transforming instances of Problem~\ref{2-in-4} into instances of Problem \ref{plshcprob}, using the reduction given in Lemma~\ref{plsred}, and then calling the oracle. Since Problem \ref{2-in-4} is NP-complete, it follows that Problem \ref{plshcprob} is NP-hard.


There are a number of variants of Problem~\ref{plshcprob} that can be shown to be NP-hard using the reduction given in Lemma~\ref{plsred}. For example, for any $\epsilon > 0$ we have the following two problems.

\begin{problem} Let $P$ be a partial latin square that satisfies \hc\ where the proportion of empty cells is less than $\epsilon$. Decide if $P$ is completable. \label{plsepsilone} \end{problem}

\begin{problem} Let $P$ be a partial latin square that satisfies \hc\ where the proportion of empty cells is greater than $1 - \epsilon$. Decide if $P$ is completable. \label{plsepsilonf} \end{problem}

In fact both these problems are NP-hard.

\begin{theorem} Problems \ref{plsepsilone} and \ref{plsepsilonf} are NP-hard. \label{epsilontheorem} \end{theorem}

\begin{proof} In Lemma~\ref{plsred}, we gave a reduction from Problem~\ref{2-in-4} to Problem~\ref{plshcprob}, that maps a 4-regular 4-uniform hypergraph $H$ on $u$ vertices to an L-shaped partial latin square of order $n = 4u^2 + 12u$, where the upper left $4u \times 4u$ square is empty. So all but a finite number of hypergraphs are mapped to partial latin squares for which the proportion of empty cells is less than $\epsilon$. It follows that Problem~\ref{plsepsilone} is NP-hard.

We can also consider a variant of the reduction given in Lemma~\ref{plsred}, where after constructing the partial latin square we delete the symbols in the bottom right $(n - 4u) \times (n - 4u)$ square. It is not difficult to see that if some symbols are deleted from a partial latin square that satisfies \hc, the resulting partial latin square must also satisfy \hc. So this reduction maps 4-regular 4-uniform hypergraphs to partial latin squares that satisfy \hc, and all but a finite number of hypergraphs are mapped to partial latin squares for which the proportion of empty cells is greater than $1 - \epsilon$. Hence Problem~\ref{plsepsilonf} is also NP-hard. \end{proof}




\section{Concluding remarks}

The most obvious open question is whether Problem~\ref{plshcprob} is in NP, and indeed, whether it is in P. \hc\ can also be defined for graphs whose vertices are equipped with colour lists (see e.g. \cite{hoffmanjohnson}). Problem~\ref{plshcprob} is a special case of the following problem.

\begin{problem} Let $G$ be a perfect graph (see e.g. \cite{bondymurty}) whose vertices have colour lists. Decide if $G$ satisfies \hc. 
\end{problem}

It would also be interesting to find more classes of partial latin square for which \hc\ is a necessary and sufficient condition for completability. This has been done to some extent in \cite{bghj}, but there are no doubt more results that could be obtained along these lines. For example, \hc\ may be a necessary and sufficient condition for completability in the case of a partial latin square for which the filled cells form a rectangle with a $2 \times t$ rectangle of empty cells inside.

\end{document}